\documentclass{amsart}
\title{Graded Tambara functors}
\author{Vigleik Angeltveit}
\address{Mathematical Sciences Institute \\
Australian National University \\
Canberra, ACT 0200 \\
Australia}
\author{Anna Marie Bohmann}
\address{Department of Mathematics\\
Vanderbilt University\\
1326 Stevenson Center\\
Nashville, TN\\
USA}

\usepackage{amsxtra}
\usepackage{amsfonts}
\usepackage[latin1]{inputenc}
\usepackage{graphicx}
\usepackage{amsmath,amssymb,latexsym,amsthm,mathrsfs}
\usepackage[all]{xy}
\usepackage{verbatim}
\usepackage{mathtools}

\newtheorem{theorem}{Theorem}[section]
\newtheorem{thm}[theorem]{Theorem}
\newtheorem{lemma}[theorem]{Lemma}

\newtheorem{prop}[theorem]{Proposition}

\theoremstyle{definition}
\newtheorem{defn}[theorem]{Definition}
\newtheorem{remark}[theorem]{Remark}
\newtheorem{example}[theorem]{Example}
\newtheorem{notation}[theorem]{Notation}

\makeatletter
\let\c@equation\c@theorem
\makeatother
\numberwithin{equation}{section}

 \newcommand{\cB}{\mathcal{B}} \newcommand{\cC}{\mathcal{C}}  \newcommand{\cE}{\mathcal{E}} \newcommand{\cF}{\mathcal{F}}   \newcommand{\cI}{\mathcal{I}} \newcommand{\cJ}{\mathcal{J}}     \newcommand{\cO}{\mathcal{O}} 
    \newcommand{\cU}{\mathcal{U}}     

 \newcommand{\bR}{\mathbb{R}}        \newcommand{\bZ}{\mathbb{Z}}

\newcommand{\Top}{\mathcal{T}} 
\newcommand{\sma}{\wedge} 

\newcommand{\inv}{{-1}}
\newcommand{\xto}{\xrightarrow}
\newcommand{\xfrom}{\xleftarrow}
\newcommand{\To}{\Rightarrow}

\newcommand{\Too}{\Longrightarrow}
\newcommand{\xToo}[1]{\xRightarrow{#1}}

\DeclareMathOperator*{\colim}{\textnormal{colim}}
\DeclareMathOperator{\Fun}{Fun}

\newcommand{\Ab}{\mathit{Ab}}

\newcommand{\RO}{\mathcal{RO}}

\newcommand{\fancyI}{{}\widehat{\cI}} 
\newcommand{\Spec}{\mathit{Sp}} 
\newcommand{\hoSpec}{ho\mathit{Sp}} 
\newcommand{\hoGSpec}{hoG\text{-}\mathit{Sp}}
\newcommand{\op}{{\mathrm{op}}}
\newcommand{\const}{\mathrm{const}}
\newcommand{\id}{\mathrm{id}}
\newcommand{\upis}{\underline{\pi}_\star}
\newcommand{\constant}{\textnormal{const}}
\newcommand{\Tambara}{\textnormal{Tamb}}
\newcommand{\Mackey}{\textnormal{Mack}}
\newcommand{\otherE}{\cE}
\newcommand{\Set}{\mathit{Set}}

\begin{document}

\begin{abstract}
We define the notion of an $\RO(G)$-graded Tambara functor and prove that any $G$-spectrum with \emph{norm multiplication} gives rise to such an $\RO(G)$-graded Tambara functor. 
\end{abstract}

\maketitle

\section{Introduction}

Let $G$ be a finite group.  The basic algebraic concept arising in $G$-equivariant homotopy theory is that of a \emph{$G$-Mackey functor}.  Mackey functors are well studied and their important role in equivariant homotopy theory has been documented since the 1970s---see, e.g.\,, \cite{Dr71} or \cite{tD73}.  For example, the appropriate $G$-equivariant version of cohomology with coefficients in an abelian group is cohomology with coefficients in a $G$-Mackey functor.  Such a cohomology theory is represented by a $G$-equivariant Eilenberg--MacLane spectrum.

A commutative multiplication on this type of $G$-cohomology theory produces a more complicated algebraic structure, called a \emph{Tambara functor.}  This type of structure arises from a commutative $G$-equivariant Eilenberg--MacLane ring spectrum \cite{Ull13}.  Compared with Mackey functors, Tambara functors have additional structure.  Loosely speaking, a Mackey functor $M$ consists of an abelian group $M(G/H)$ for each subgroup $H\leq G$, together with restriction and transfer maps between these groups that satisfy certain relations.  A Tambara functor has an additional type of map, called a ``norm map,'' relating the groups $M(G/H)$.  One can think of restriction as an equivariant version of a diagonal map, transfer as an equivariant version of addition, and norm as an equivariant version of multiplication.

The theory of Tambara functors, which was first introduced by Tambara in \cite{Ta93} under the name of TNR-functors, is not as well developed as that of Mackey functors. Brun \cite{Br07} and Strickland \cite{St} both discuss  Tambara functors with an eye towards homotopy theory.  Both prove that the zeroth homotopy groups of an $E_\infty$-ring spectrum form a Tambara functor.  However, neither work considers the algebraic structure present on the homotopy groups in nonzero grading.  The main goal of the present work is to understand this structure, which is that of a \emph{graded} Tambara functor.  Thus, our results can be thought of as generalizing Strickland's and Brun's work.   At the Mackey functor level, the analogous results are well known, but our work requires a refinement of the existing treatment of \emph{graded} Mackey functors. To set the stage for our graded Tambara functors we define the notion of an $\RO(G)$-graded Mackey functor, and prove that any $G$-spectrum $E$ determines such an $\RO(G)$-
graded Mackey functor in Theorem \ref{t:Mackey}. Here $\RO(G)$ is a categorification of the real representation ring of $G$, and our first task is to define $\RO(G)$ carefully.

Additionally, the literature  suggests that $E$ has to be a $G$-equivariant $E_\infty$ ring spectrum in order for its homotopy groups to define a Tambara functor, see e.g.\ the first paragraph on p.\ 235 of \cite{Br05}.
This condition seems stronger than necessary, as in the case $G=\{e\}$ a homotopy associative and commutative multiplication clearly suffices to give $\pi_* E$ the structure of a graded commutative ring. We remedy this situation by defining the notion of a \emph{norm multiplication} on a $G$-spectrum, and prove in Theorem \ref{t:Tambara} that if $E$ has a norm multiplication then its $\RO(G)$-graded homotopy groups constitute a graded Tambara functor.

\subsection{Statement of results}
Our first contribution is a precise definition of the categorified representation ring. Given a finite $G$-set $X$, we make the following definition.

\begin{defn}
Let $\RO(G)(X)$ denote the category $Fun(\cB_G X, \fancyI^\op)$ of functors from the translation category $\cB_G X$ to the Grayson-Quillen construction on the category of finite dimensional real inner product spaces. The morphisms are natural transformations.
\end{defn}

As a first step we define a category $\RO(G)$ by declaring a morphism from $(X,\chi)$ to $(Y,\gamma)$ to be a pair $(f,\tilde{f})$ where
\[
 f\colon X \to Y
\]
is an isomorphism of $G$-sets and
\[
 \tilde{f}\colon \chi \Rightarrow \gamma \circ f
\]
is a natural transformation of functors from $\cB_G X$ to $\fancyI^\op$.

We go on to define a category $\RO(G)^\Mackey$ by adding restriction and transfer maps to $\RO(G)$, and define an $\RO(G)$-graded Mackey functor to be a functor $\RO(G)^\Mackey \to \Ab$ to the category of abelian groups. The following result is also restated as Theorem \ref{t:Mackey_restated}:

\begin{thm} \label{t:Mackey}
Let $E$ be an orthogonal $G$-spectrum. Then $E$ determines an $\RO(G)$-graded Mackey functor
\[
 \upis(E)\colon \RO(G)^\Mackey \to \Ab.
\]
\end{thm}

Next we define a category $\RO(G)^\Tambara$ by also adding norm maps to $\RO(G)$, and define an $\RO(G)$-graded Tambara functor to be a functor $\RO(G)^\Tambara \to \Ab$. We then pin down the exact amount of multiplicative structure $E$ needs in order for $\upis(E)$ to determine an $\RO(G)$-graded Tambara functor.

\begin{defn}
A \emph{norm multiplication} on a $G$-spectrum $E$ is a natural transformation
\[
 \mu_A^B\colon n_*^\sma \constant_A(E) \Rightarrow \constant_B(E)
\]
in the homotopy category of functors $\cB_G B \Rightarrow \Spec$ for each $n\colon A \to B$ satisfying the properties listed in Definition \ref{d:normmult} below.
\end{defn}

For the definition of $n_*^\sma$ see Section \ref{ss:wedgeandsmash}. A norm multiplication on $E$ amounts to a map $N_H^K E \to E$ for each $H \leq K$, where $N_H^K$ is the Hill--Hopkins--Ravenel norm \cite{HHR}, together with a map $E \sma E \to E$. These then have to satisfy certain compatibility axioms. If $E$ is a commutative orthogonal $G$-spectrum then the ring structure determines a norm multiplication, but our conditions are weaker. If $G=\{e\}$ is the trivial group a norm multiplication is the same as a multiplication map that is homotopy unital, associative and commutative. The following result is also restated as Theorem \ref{t:Tambara_restated}.

\begin{thm} \label{t:Tambara}
Let $E$ be an orthogonal $G$-spectrum with a norm multiplication. Then $E$ determines an $\RO(G)$-graded Tambara functor
\[
 \upis(E)\colon \RO(G)^\Tambara \to \Ab.
\]
\end{thm}

\subsection{Acknowledgments}
The authors would like to thank Mike Hill for interesting conversations about equivariant stable homotopy theory. This project was started after Discovery Grant No.\ DP120101399 from the Australian Research Council paid for the second author to visit the first author. We would also like to thank MSRI for hosting both of us during the Algebraic Topology program where some of this work was carried out.  This material is based on work partially supported by the National Science Foundation under Grant No.\ 0932078 000 while the authors were in residence at the Mathematical Sciences Research Institute in Berkeley, California, during the Spring 2014 semester.  The authors also thank an anonymous referee for comments which improved and clarified this work.

\section{Orthogonal $G$-spectra} \label{s:orthogonalspectra}
\subsection{Orthogonal spectra}
We will give a very brief review of orthogonal spectra; see e.g.\ \cite{MM02} or \cite[Appendix A and B]{HHR} for a more complete description. Recall that we can define an orthogonal spectrum as follows. Let $\cI$ denote the category of finite dimensional real inner product spaces and linear isometric embeddings. Then an orthogonal spectrum is a functor $E\colon \cI \to \Top$ to the category of based spaces which is in addition a module over the sphere spectrum. A map of orthogonal spectra is a natural transformation of functors. We denote the category of orthogonal spectra by $\Spec$.

To spell this out, an orthogonal spectrum $E$ consists of a based space $E(V)$ with an $O(V)$-action for each finite dimensional real inner product space $V$, together with structure maps $S^V \sma E(W) \to E(V \oplus W)$ which are $O(V) \times O(W)$-equivariant. Note that given $E(V)$ and some $V' \in \cI$ with $V \cong V'$, $E(V')$ is determined up to homeomorphism by $E(V)$.

A map $f\colon E \to E'$ of orthogonal spectra is a map $f(V)\colon E(V) \to E'(V)$ for each $V$ which commutes with all the structure maps.

There is a neat way to package a functor $E\colon \cI \to \Top$ which is also a module over the sphere spectrum, namely as a functor $\cJ \to \Top$. Here $\cJ$ is the category with the same objects as $\cI$, and $\cJ(V,U)$ is defined as follows. Consider the space $\cI(V,U)$ of linear isometric embeddings. It has a normal bundle, and $\cJ(V,U)$ is the Thom space of this normal bundle. In symbols we have
\[
 \cJ(V,U) = \{(f,u) \mid f\colon V \to U, \,\,\, u \in U-f(V) \} \cup \{\infty\}.
\]

\begin{example}
Let $V$ be a finite dimensional real inner product space and let $X$ be a based space. Then there is an orthogonal spectrum $\cF_V(X)$ defined by
\[
 \cF_V(X)(U) = \cJ(V,U) \sma X.
\]
The functor $\cF_V(-)$ is left adjoint to the functor $ev_V\colon \Spec \to \Top$ which picks out the $V$'th space. If $V=0$ then $\cF_V(X)$ is the usual suspension spectrum of $X$.
\end{example}

The spectrum $\cF_V(S^V)$ plays an important role. We construct the (positive) stable model structure on $Sp$ from the (positive) level model structure by formally declaring the maps $\cF_V(S^V) \to S=\cF_0(S^0) = \Sigma^\infty S^0$ to be weak equivalences. 

\subsection{Equivariant orthogonal spectra}\label{subsec:equivorthospec}
Now let $G$ be either a finite group or a compact Lie group. (Later we will restrict our attention to finite groups.) We define a $G$-equivariant orthogonal spectrum simply as an orthogonal spectrum $E$ with a $G$-action, and we denote the category of $G$-equivariant orthogonal spectra by $G\text{-}\Spec$.

To the reader who is used to dealing with universes this might seem like a naive definition, but by varying the model structure we recover both the naive and genuine homotopy categories. The point is that if $V$ is an $n$-dimensional $G$-representation we can define $E(V)$ by
\[
 E(V) = \cI(\bR^n, V)_+ \sma_{O(n)} E(\bR^n),
\]
with $G$ acting on the target of $\cI(\bR^n, V)$ and diagonally on the smash product.

This is non-equivariantly homeomorphic to $E(\bR^n)$. Given a universe $\cU$ we can define the homotopy groups of $E$ by taking a colimit over the finite-dimensional subspaces of $\cU$. For example, 
\[
 \pi_0^G E = \colim_V [S^V, E(V)]_G.
\]

A map of $G$-spectra is a weak equivalence if it induces an isomorphism of $H$-equivariant homotopy groups for each closed $H \leq G$. When discussing homotopy classes of maps between orthogonal $G$-spectra we  always work with the model structure corresponding to a complete $G$-universe.

\subsection{Homotopy groups and the Pontryagin--Thom construction}
Here is one way to define the homotopy groups of a spectrum $E$. Let $q \in \bZ$, and write $q=m-n$ for $m, n \geq 0$. Then we can define
\[
 \pi_q(E) = [S^m, S^n \sma E]
\]
where $[-,-]$ denotes maps in $\hoSpec$.

Similarly, if $E$ is a $G$-spectrum and $(V,W)$ is a pair of $G$-representations we can define
\[
 \pi_{V-W}^G(E) = [S^V, S^W \sma E]_G
\]
where $[-,-]_G$ denotes maps in $\hoGSpec$.

We now briefly explain how to use the Pontryagin--Thom construction to define addition in the homotopy groups of $E$. Suppose $\dim(V^G) > 0$. Then for any $n \geq 0$ we can find a $G$-equivariant embedding of $\{1,2,\ldots,n\}$ in $V$. Here $G$ acts trivially on $\{1,2,\ldots,n\}$. It follows that for $\epsilon$ sufficiently small the disjoint union of $n$ balls of radius $\epsilon$ injects $G$-equivariantly in $V$.

Given these embedded $\epsilon$-balls, we get a Pontryagin--Thom collapse map $S^V \to \bigvee_n S^V$, and precomposing with this gives a map
\[
 \big( \pi_{V-W}^G(E) \big)^n \to \pi_{V-W}^G(E).
\]
This is the addition map, which adds up $n$ elements in $\pi_{V-W}^G(E)$ using the abelian group structure.

As we will explain in Section \ref{s:Mackey} below the transfer map is a generalization of this idea. This is of course well known, but precision here helps clarify the overall structure.

\section{The categorified representation ring} \label{s:repring}
\subsection{The Grayson-Quillen construction}
Let $G$ be a compact Lie group, and let $RO(G)$ denote the real representation ring of $G$. If $\alpha \in RO(G)$ we can choose $G$-represen\-ta\-tions $V$ and $W$ with $\alpha = [V]-[W]$, and if $E$ is a $G$-spectrum we can attempt to define $\pi_\alpha^G(E) = [S^V, S^W \sma E]_G$
to be the set of homotopy classes of $G$-equivariant maps as above. The problem is that this is only well defined up to non-canonical isomorphism, and we would like to avoid having to make arbitrary choices. See \cite{Du14} for a related discussion and a proof that it is possible to make all choices necessary to define $\pi_\alpha^G(E)$ for each $\alpha \in RO(G)$ coherently.

Instead we define a category whose objects are actual pairs $(V,W)$ of $G$-re\-pre\-sen\-ta\-tions. In fact, we will start with a non-equivariant category of pairs and introduce equivariance by considering functors from the category $\cB_G *$ with one object and morphism set $G$.

As in the definition of orthogonal spectra, let $\cI$ be the category of finite dimensional real inner product spaces, with morphisms given by linear isometric embeddings.

In analogy with the construction in \cite[\S 4]{SaSc12} we define the Grayson-Quillen category $\fancyI$ as follows. This category also appears in \cite[Section 3]{Kr07} and \cite[Definition 2.6]{SaSc14}.

\begin{defn}
The category $\fancyI$ has objects pairs $(V,W)$ of finite dimensional inner product spaces. A morphism $(V_1,W_1) \to (V_2,W_2)$ is a triple $(f,g,\phi)$, where $f\colon V_1 \to V_2$ and $g\colon W_1 \to W_2$ are maps in $\cI$ and $\phi\colon \big( V_2 - f(V_1) \big) \to \big( W_2 - g(W_1) \big)$ is an isomorphism in $\cI$.
\end{defn}

To obtain a $G$-action we consider the category $\cB_G *$ with one object $*$ and morphism set $G$. If $\cC$ is any category, giving a functor from $\cB_G *$ to $\cC$ is precisely the same data as an object in $\cC$ together with a $G$-action on that object. A natural transformation of functors is precisely the same data as a $G$-equivariant morphism between $G$-objects in $\cC$. 

We claim that the category $Fun(\cB_G *, \fancyI^\op)$ of $G$-objects in $\fancyI^\op$ is a categorification of the representation ring $RO(G)$. We use the opposite category of $\fancyI$ because we can then use Lemma \ref{l:ItoSpec} below to get a functor from $\cB_G *$ to spectra.

There is a ``connected components'' map $Fun(\cB_G *, \fancyI^\op) \to RO(G)$ which takes a pair $(V,W) \in Fun(\cB_G *, \fancyI^\op)$ to the equivalence class $[V]-[W]\in RO(G)$.  The existence of a map $(V_1,W_1)\to (V_2,W_2)$ in $\fancyI^\op$ implies that the equivalence classes $[V_1]-[W_1]$ and $[V_2]-[W_2]$ are equal in $RO(G)$, and thus we regard this map as taking $\pi_0$ of $Fun(\cB_G *, \fancyI^\op)$.

The category $\fancyI^\op$ fits well with the category of orthogonal spectra, as the following lemma (which also appears as \cite[Lemma 4.3]{SaSc14}) shows.

\begin{lemma} \label{l:ItoSpec}
There is a functor $\cF\colon \fancyI^{op} \to Sp$ given on objects by $(V,W) \mapsto \cF_W(S^V)$. This functor takes any map in $\fancyI^{op}$ to a stable equivalence in $Sp$.
\end{lemma}

\begin{remark}
We will denote $\cF(V,W)$ by $S^{V-W}$, remembering that $S^{V-W}$ is only weakly equivalent to $S^{(V \oplus U)-(W \oplus U)}$. In \cite{HHR} the spectrum $S^{V-W}$ is denoted $S^{-W} \sma S^V$.
\end{remark}

\begin{proof}
Given a map
\[
 (f,g,\phi)\colon (V_1,W_1) \to (V_2,W_2)
\]
in $\fancyI$ we need to produce a map
\[
 S^{V_2-W_2} \to S^{V_1-W_1}
\]
of spectra. Such a map is adjoint to a map
\[
 S^{V_2} \to \cF_{W_1}(S^{V_1})(W_2) = \cJ(W_1,W_2) \sma S^{V_1}
\]
of spaces. Using $(f,g,\phi)$ we get a homeomorphism
\[
 S^{V_2} \xto{\cong} S^{W_2-g(W_1)} \sma S^{V_1},
\]
and we recall that the Thom space $\cJ(W_1,W_2)$ is given by
\[
 \cJ(W_1,W_2) = \{(h\colon W_1 \to W_2, w \in W_2-h(W_1))\} \cup \{\infty\}.
\]
We have a map
\[
 i_g\colon S^{W_2-g(W_1)} \to \cJ(W_1,W_2)
\]
defined by $w \mapsto (g,w)$ for $w \neq \infty$ and $\infty \mapsto \infty$. Hence we get a composite map
\[
 S^{V_2} \xto{\cong} S^{W_2-g(W_1)} \sma S^{V_1} \xto{i_h \sma id} \cJ(W_1,W_2) \sma S^{V_1}
\]
as required.

We leave it to the reader to check that this is indeed a functor from $\fancyI^{op}$ to $Sp$. It is clear that $S^{V_2-W_2} \to S^{V_1-W_1}$ is a stable equivalence because we localized with respect to those maps when defining the stable model category structure on $Sp$.
\end{proof}

Recall that $\cB_G *$ denotes the category with one object and morphisms set $G$.

\begin{defn}
We denote the category $Fun(\cB_G *, \fancyI^\op)$ of functors from $\cB_G *$ to $\fancyI^\op$ and natural transformations between them by $\RO(G)(\cB_G *)$ or $\RO(G)(*)$.
\end{defn}

Now suppose $G$ is a finite group and $X$ is a finite $G$-set. We let $\cB_G X$ denote the category with object set $X$, and with a morphism $x \to y$ for each $g \in G$ with $gx = y$. This category is sometimes called the ``translation category'' of $G$ acting on $X$. Then we can generalize the above definition as follows.

\begin{defn}
We denote the category $Fun(\cB_G X, \fancyI^\op)$ of functors from $\cB_G X$ to $\fancyI^\op$ and natural transformations between them by $\RO(G)(\cB_G X)$ or $\RO(G)(X)$.
\end{defn}

Composition with $\cF$ yields a map
\[
 \RO(G)(X) = Fun(\cB_G X, \fancyI^\op) \to Fun(\cB_G X, Sp).
\]
\begin{notation}
Given a functor $\chi\colon \cB_G X \to \fancyI^\op$ we will write $S^\chi$ for the composite $\cB_G X \to \fancyI^\op \to Sp$.
\end{notation}

\subsection{Wedge sums and smash products} \label{ss:wedgeandsmash}
For any finite $G$-set $X$, there are two functors
\[
 p_*^\vee, p_*^\sma\colon Fun(\cB_G X, Sp) \to G\text{-}\Spec
\]
which we will now describe. The first takes $F\colon B_G X \to \Spec$ to
\[
 p_*^\vee(F) = \bigvee_{x \in X} F(x)
\]
and the second takes $F$ to
\[
 p_*^\sma(F) = \bigwedge_{x \in X} F(x).
\]
The $G$-action is given by permuting the wedge sums or smash factors using the maps $F(g)\colon F(x) \to F(gx)$ coming from the functoriality of $F$. These constructions are examples of the ``indexed monoidal products'' discussed in \cite{HHR}.

More generally, suppose $f\colon X \to Y$ is a map of $G$-sets and let $f$ also denote the functor $\cB_G X \to \cB_G Y$ of translation categories. Then we have functors
\[
 f_*^\vee, f_*^\sma\colon Fun(\cB_G X, \Spec) \to Fun(\cB_G Y, \Spec)
\]
defined by
\[
 (f_*^\vee F)(y) = \bigvee_{f(x)=y} F(x)
\]
and
\[
 (f_*^\sma F)(y) = \bigwedge_{f(x)=y} F(x).
\]
The functors $p_*^\vee$ and $p_*^\sma$ are induced by the map $p\colon X \to *$.

As one might expect, these functors participate in adjunctions. Denote by $f^*\colon Fun(\cB_G Y, \Spec) \to Fun(\cB_G X, \Spec)$ the functor induced by precomposition.

\begin{prop} \label{p:wedgead}
Let $f\colon X \to Y$ be a map of finite $G$-sets. Then there is an adjunction
\[
 f_*^\vee\colon Fun(\cB_G X, \Spec) \rightleftarrows Fun(\cB_G Y, \Spec)\colon f^*.
\]
\end{prop}

To get the corresponding adjunction for the smash product we restrict our attention to commutative ring spectra. Let $\cC\Spec$ denote the category of commutative orthogonal ring spectra. Then $f_*^\sma$ and $f^*$ restrict to functors with target $\cC\Spec$.

\begin{prop} \label{p:smashad}
Let $f\colon X \to Y$ be a map of finite $G$-sets. Then there is an adjunction
\[
 f_*^\sma\colon Fun(\cB_G X, \cC\Spec) \rightleftarrows Fun(\cB_G Y, \cC\Spec)\colon f^*.
\]
\end{prop}

The functor $f_*^\sma$ is essentially the Hill--Hopkins--Ravenel norm; see Definition \ref{d:norm} below.

\subsection{An equivalence of categories}\label{subsect:equivalenceofcat}
We pause to remark on an equivalence of categories that allows us to relate our constructions to more familiar definitions. If $X=G/H$, we have equivalences
\[
\xymatrix{
\cB_G G/H  \ar@<.5ex>[r]^-{\kappa}& \ar@<.5ex>[l]^-{\iota} \cB_H * 
}
\]
Here $\iota$ is given by the inclusion of the identity coset, while $\kappa$ depends on a choice of coset representatives $g_1 H,\ldots,g_n H$ for $G/H$ with $g_1 H = eH$. Given such a choice, $\kappa$ is defined as follows. For each $g\colon g_i H \to g_j H$ there is a unique way to write $gg_i = g_j h$ with $h \in H$, and we send $g\colon g_i H \to g_j H$ to $h\colon * \to *$. Then $\kappa \circ \iota = id_{\cB_H *}$, while there is a natural transformation $id \Rightarrow \iota \circ \kappa$ given by $g_i^{-1}\colon g_i H \to eH$.

A functor $F\colon \cB_G G/H \to \Spec$ consists of a spectrum $F(g_i H) = E_{g_i H}$ for each coset, together with a map $g\colon E_{g_i H} \to E_{g_j H}$ for each $g$ with $gg_i = g_j h$. Precomposing with $\iota$ has the effect of only remembering $E_{eH}$ and forgetting the other $E_{g_i H}$.
Conversely, a functor $F'\colon \cB_H * \to Sp$ consists of an $H$-spectrum $E'$. Precomposing with $\kappa$ has the effect of producing a functor $F\colon \cB_G G/H \to Sp$ with $F(g_i H) = E'$ for each $i$ and with $G$-action as defined as above.

If $X$ is isomorphic to $G/H$, where $G/H$ has cosets $g_i H$, then $X$ is also isomorphic to $G/H_i$ with $H_i=g_i H g_i^{-1}$. Using the inclusion $\cB_{H_i} * \to B_G X$ we get a different equivalence of categories where we remember $E_{g_i H}$ as an $H_i$-spectrum rather than $E_{eH}$ as an $H$-spectrum.
We can interpret this as saying that given a $G$-set $X$ and a functor $F\colon \cB_G X \to \Spec$, we have to choose a ``basepoint'' of each orbit in order to identify $F$ with a collection of $H_i$-equivariant spectra.

This equivalence of categories allows us to understand homotopy classes of maps in the functor categories $Fun(\cB_GX,\Spec)$.  When $X$ is a point, $Fun(\cB_GX,\Spec)$ is simply the category of $G$-spectra, and we endow it with the genuine model structure.  For an orbit $G/H$, the category $Fun(\cB_GG/H, \Spec)$ is equivalent to $Fun(\cB_H*,\Spec)$, the category of $H$-spectra.  Hence $Fun(\cB_GG/H,\Spec)$ inherits the genuine $H$-equivariant model structure from the category of $H$-spectra.  For a more general $G$-set $X$, we can view the category$Fun(\cB_GX,\Spec)$ as the product of categories of the form $Fun(\cB_G G/H_i,\Spec)$ via a decomposition of $X$ into orbits.  This endows $Fun(\cB_GX,\Spec)$ with the model structure given by taking the product of the genuine $H_i$ model structures.  This process is independent of the chosen decomposition of $X$ into orbits; see \cite[\S B.5]{HHR} where they also provide the following lemmas.
\begin{lemma}[{\cite[\S B.5.1]{HHR}}] The adjunction of Proposition \ref{p:wedgead} is a Quillen adjunction. 
\end{lemma}
\begin{lemma}[{\cite[Proposition B.104]{HHR}}]
The adjunction in  Proposition \ref{p:smashad} is also a Quillen adjunction.  In fact, the functor $f^\sma_*$ further descends to a left derived functor between the homotopy categories $Fun(\cB_GX,\Spec)$ and $Fun(\cB_GY,\Spec)$.  
\end{lemma}
This result is an important technical underpinning to the Hill--Hopkins--Ravenel norm.   For us, this fact will permit us to use the functor smash-induction functor $f^\sma_*$ even when dealing with spectra that do not enjoy a commutative ring structure, which we will do throughout the remainder of this paper.

To illustrate how these definitions fit into the existing literature, we next define the Hill--Hopkins--Ravenel norm.

\begin{defn}[{\cite{HHR}}] \label{d:norm}
Given an $H$-spectrum $E$ and a choice of $\kappa\colon B_G G/H \to B_H *$ as above, define the transfer $T_H^G(E)$ to be
\[
 T_H^G(E) = p_*^\vee(\kappa^*(E))
\]
and define the norm $N_H^G(E)$ to be
\[
 N_H^G(E) = p_*^\sma(\kappa^*(E)).
\]
\end{defn}

Here $T_H^G(E) = \bigvee_{g_i H} E$, and we have an equivalence
\[
 T_H^G(E) \to G_+ \sma_H E
\]
given by sending the wedge summand of $E$ indexed by $g_i H$ to $[g_i, E]$. Hence $T_H^G$ is equivalent to the usual induction functor. 

\subsection{Equivariant homotopy groups}
Above we defined a category $\RO(G)(X)$ for each finite $G$-set $X$, and in the next two sections we will make $\RO(G)$ into a category in two interesting ways. For now we make $\RO(G)$ into a category by defining $Hom_{\RO(G)}((X,\chi), (Y,\gamma))$ to be the set of pairs $(f,\tilde{f})$ where $f\colon X \to Y$ is an isomorphism and $\tilde{f}\colon \chi \Rightarrow f^* \gamma$ is a natural transformation of functors.

Composition is defined in the obvious way: Given $(f,\tilde{f})\colon (X,\chi) \to (Y,\gamma)$ and $(g,\tilde{g})\colon (Y,\gamma) \to (Z, \zeta)$, the composite $(g, \tilde{g}) \circ (f,\tilde{f})$ is defined as $(g \circ f, f^*(\tilde{g}) \circ \tilde{f})$.

For each $(X,\chi) \in \RO(G)$ and orthogonal $G$-spectrum $E$ we can define the $\RO(G)$-graded homotopy group
\[
 \upis(E)(X,\chi) = [\bigvee_{x \in X} S^{\chi(x)}, E]_G = [p_*^\vee(S^\chi), E]_G
\]
to be the set of homotopy classes of $G$-equivariant maps from the wedge of spheres determined by $(X,\chi)$.

Given a map $(f, \tilde{f})\colon (X, \chi) \to (Y,\gamma)$ in $\RO(G)$, we can apply $\cF(-)$ to $\tilde{f}$ to get a natural transformation $\cF(\tilde{f})\colon S^\chi \to S^{f^* \gamma}$. Moreover, because $f$ is an isomorphism the $G$-spectra $p_*^\vee(S^{f^* \gamma})$ and $p_*^\vee(S^\gamma)$ are canonically isomorphic. Hence we have a map
\[
 p_*^\vee(S^\chi) \xto{\tilde{f}} p_*^\vee(S^{f^* \gamma}) \xto{\cong} p_*^\vee(S^\gamma).
\]
By precomposing with this we get a map
\[
 (f,\tilde{f})^*\colon \upis(E)(Y,\gamma) \to \upis(E)(X,\chi).
\]
This is clearly compatible with composition in $\RO(G)$, so we get a functor
\[
 \upis(E)\colon \RO(G)^\op \to Ab.
\]

In Section \ref{s:Mackey} we will add maps to the category $\RO(G)^\op$ to define a category $\RO(G)^\Mackey$, so that $\upis(E)$ defines a functor from $\RO(G)^\Mackey$ to $\Ab$. This will be our notion of a \emph{graded Mackey functor}.

In Section \ref{s:Tambara} we will add additional maps to $\RO(G)^\Mackey$ to define a category $\RO(G)^{\Tambara}$, and show that if $E$ has the appropriate multiplicative structure then $\upis(E)$ defines a functor from $\RO(G)^{\Tambara}$ to $\Ab$. This will be our notion of a \emph{graded Tambara functor}.

\subsection{The constant functor at $E$}
If $E$ is a $G$-spectrum and $X$ is a finite $G$-set we can define a functor
\[
 \const_X(E)\colon \cB_G X \to \Spec
\]
by sending any $x$ to $E$.  A map $g\colon x\to y$ is sent to the map $E\to E$ given by the action of $g\in G$.

Above we defined $\upis(E)(X,\chi) = [p_*^\vee(S^\chi), E]_G$, but by the adjunction in Proposition \ref{p:wedgead} this is naturally isomorphic to $[S^\chi, \constant_X(E)]_{\cB_G X}$. Here $[-,-]_{\cB_G X}$ denotes the set of homotopy classes of natural transformations of functors from $\cB_G X$, defined using the model structure of \S\ref{subsect:equivalenceofcat}. This simply says that giving a $G$-equivariant map from $\bigvee_{x \in X} S^{\chi(x)}$ to $E$ is the same as giving a $G$-equivariant map from $\bigvee_{x \in X} S^{\chi(x)}$ to $\bigvee_{x \in X} E$ sending the sphere indexed by $x$ to the copy of $E$ indexed by $x$.

Given $(f, \tilde{f})\colon (X, \chi) \to (Y,\gamma)$, note that $f^*$ gives a map from the set of (homotopy classes of) natural transformations of functors from $Y$ to the set of (homotopy classes of) natural transformations from $X$ and that $f^* \const_Y(E) = \const_X(E)$. So from this point of view the induced map $(f,\tilde{f})^*\colon \upis(E)(Y,\gamma) \to \upis(E)(X,\chi)$ is given by the composite
\[
 [S^\gamma, \const_Y(E)]_{\cB_G Y} \xto{f^*} [S^{f^* \gamma}, \const_X(E)]_{\cB_G X} \xto{\tilde{f}^*} [S^\chi, \const_X(E)]_{\cB_G X}.
\]
We will use this repeatedly in the rest of the paper.

\subsection{Relation to usual homotopy groups}\label{subsect:relatingtoHhtpygrps}
If $X$ is an orbit $G/H$, the $\RO(G)$-graded homotopy group $\upis(E)(X,\chi)$ can be identified with an $H$-equivariant homotopy group. As above we have a weak equivalence
\[
 p_*^\vee(S^\chi) = \bigvee_{gH \in G/H} S^{\chi(gH)} \simeq G_+\sma_H S^{\chi(eH)}
\]
given on the wedge summand indexed by $gH$ by mapping to $g$ in the first smash factor and by using the map $g^{-1}\colon S^{\chi(gH)} \to S^{\chi(eH)}$ on the second smash factor. (By abuse of notation one might want to write $v \mapsto [g, g^{-1} v]$.)

The standard change of groups adjunctions then apply to show that
\[
 [\bigvee_{gH \in G/H} S^{\chi(gH)}, E]_G \cong [S^{\chi(eH)}, E]_H
\]
where on the right hand side we have regarded the $G$-spectrum $E$ as an $H$-spectrum by restricting the action.  Notice that $S^{\chi(eH)}$ is also an $H$-spectrum because the coset $eH$ is stabilized by $H$. The group on the right hand side is the usual $\RO(H)$-graded homotopy group
\[
 \pi_{\chi([eH])}^H(E).
\]

\section{Graded spans and graded Mackey functors} \label{s:Mackey}
Recall that an ordinary Mackey functor consists of an abelian group $M(X)$ for each finite $G$-set $X$, and a morphism $M(X) \to M(Y)$ for each diagram
\[
 X \xfrom{r} A \xto{t} Y.
\]
The map $r\colon A \to X$ determines a \emph{restriction map} $r^*\colon M(X) \to M(A)$ and the map $t\colon A \to Y$ determines a \emph{transfer map} $t_*^\vee\colon M(A) \to M(Y)$. (We add the superscript $\vee$ to distinguish this from the map $n_*^\sma$ we will define later.) Diagrams of this form yield a category in which composition is given by pullback. If
\[
 Y \xfrom{r'} B \xto{t'} Z
\]
is another span, the composite span is given by the pullback $D = A \times_Y B$ together with the obvious maps to $X$ and $Z$.   In this section, we define an $\RO(G)$-version of this construction.

\subsection{Restriction maps}
We define a category $\RO(G)^R$ with the same objects as $\RO(G)$ as follows.

\begin{defn}
Let $(X,\chi)$ and $(A,\alpha)$ be in $\RO(G)$. A map in $\RO(G)^R$ from  $(A,\alpha)$ to $(X,\chi)$ is a pair $(r, \tilde{r})$ where $r\colon A \to X$ is a $G$-map and $\tilde{r}\colon \alpha \Rightarrow r^* \chi$ is a natural transformation of functors.

Composition in $\RO(G)^R$ is defined in the obvious way, by composing maps of $G$-sets and natural transformations.
\end{defn}

\begin{remark}
The existence of the natural transformation $\tilde{r}$ implies that for each $a \in A$ there is a map $\tilde{r}_a\colon \alpha(a) \to \chi(r(a))$ in $\fancyI^\op$. The map $\tilde{r}_a$ gives an isomorphism of virtual vector spaces between $(V_{\chi(r(a))}, W_{\chi(r(a))})$ and $(V_{\alpha(a)}, W_{\alpha(a)})$, i.e., an isomorphism $V_{\alpha(a)} \oplus W_{\chi(r(a))} \cong V_{\chi(r(a))}  \oplus W_{\alpha(a)}$.

Taking the $G$-action into account we can say the following. Let $G_a$ denote the  stabilizer of $a$. Then $\alpha(a)$ determines a representative $[\alpha(a)] \in RO(G_a)$ and $\chi(r(a))$ determines a representative $[\chi(r(a))] \in RO(G_{r(a)})$, and it follows that
\[
 [\alpha(a)] = Res^{G_{r(a)}}_{G_a}[\chi(r(a))].
\]
In other words, if the above equation does not hold then a restriction map does not exist. Another point of view is that up to maps in $\fancyI$ the map $r\colon A \to X$ and $\chi \in \RO(G)(X)$ determines $\alpha$.

If we decompose $X$ and $A$ into orbits by specifying isomorphisms $X \cong G/H_i$ and $A \cong G/K_j$ then $\chi$ determines an element $[\chi] \in \bigoplus RO(H_i)$ and $\alpha$ determines an element $[\alpha] \in \bigoplus RO(K_j)$. If we choose the isomorphisms in such a way that $r$ corresponds to maps $eK_j \mapsto eH_i$ then $[\alpha]_j = Res^{H_i}_{K_j}([\chi]_i)$. Any other choice involves additional conjugation.
\end{remark}

By Proposition \ref{p:wedgead}, the pair $(r,\tilde{r})$ gives a $G$-equivariant map
\[
 (r,\tilde{r})_*\colon \bigvee_{a \in A} S^{\alpha(a)} \to \bigvee_{x \in X} S^{\chi(x)}.
\]
In more categorical language, the natural transformation $\cF(\tilde{r}) \in \Fun(\cB_G A,\Spec)$ is adjoint to a natural transformation $r_*^\vee \alpha \To \chi$ in $\Fun(\cB_G X,\Spec)$ which induces the displayed map on equivariant bouquets of spheres.

Given a map $(r,\tilde{r})\colon (A,\alpha) \to (X,\chi)$ in $\RO(G)^R$ and a $G$-spectrum $E$, we then get a map
\[
 (r, \tilde{r})^*\colon \upis(E)(X,\chi) \to \upis(E)(A,\alpha)
\]
by precomposing with this map of bouquets of spheres. This is clearly compatible with composition, so we have a functor
\[
 \upis(E)\colon (\RO(G)^R)^\op \to \Ab.
\]
We can also phrase this in terms of the constant spectrum $\const_A(E)$. Given an element $\phi \in [S^\chi, \const_X(E)]_{\cB_G X}$ we map it to the composite
\[
 S^\alpha \xToo{\tilde{r}} r^* S^\chi \xToo{r^* \phi} r^* \const_X(E) = \const_A(E).
\]

\begin{remark}
A map $(f, \tilde{f})\colon (X,\chi) \to (Y,\gamma)$ in $\RO(G)^R$ for which $f$ is an isomorphism is exactly the same as a map in the category $\RO(G)$ defined earlier. Hence we can view $\RO(G)^R$ as a natural generalization of $\RO(G)$.
\end{remark}

\subsection{Transfer maps}
Next we define a category $\RO(G)^T$ with the same objects as $\RO(G)$.

\begin{defn}
Let $(A,\alpha)$ and $(Y,\gamma)$ be in $\RO(G)$. A map in $\RO(G)^T$ from $(A,\alpha)$ to $(Y,\gamma)$ is a pair $(t,\tilde{t})$ where $t\colon A \to Y$ is a $G$-map and $\tilde{t}\colon \alpha \To t^* \gamma$ is a natural transformation satisfying the following conditions:
\begin{itemize}
 \item Each component $\tilde{t}_a\colon \alpha(a) \to \gamma(t(a))$ is an isomorphism in $\fancyI^{op}$. In other words, $\tilde{t}_a$ is a pair of isomorphisms $V_{\alpha(a)} \to V_{\gamma(t(a))}$ and $W_{\alpha(a)} \to W_{\gamma(t(a))}$.
 \item There exists a $G$-equivariant embedding of $A$ in $\coprod V_{\gamma(y)}$ with $a \in V_{\gamma(t(a))}$. This implies that for $\epsilon$ sufficiently small the corresponding $\epsilon$-balls are disjoint, and it implies that we can construct an injective ($G$-equivariant and continuous, but not linear) map $\coprod V_{\alpha(a)} \to \coprod V_{\gamma(y)}$ with image the disjoint union of the $\epsilon$-balls.
\end{itemize}

Composition in $\RO(G)^T$ is defined in the obvious way. The condition that each component of $\tilde{t}$ is an isomorphism is obviously stable under composition and the existence of a $G$-equivariant embedding is too.
\end{defn}

Given $(t,\tilde{t})$ in $\RO(G)^T$ the conditions on $\tilde{t}$ give us a Pontryagin--Thom collapse map
\[
 PT(\tilde{t})\colon \bigvee_{y \in Y} S^{\gamma(y)} \to \bigvee_{a \in A} S^{\alpha(a)}.
\]
In more detail, we pick a $G$-equivariant embedding of $A$ in $\coprod V_{\gamma(y)}$ and an $\epsilon$ such that the $\epsilon$-balls are disjoint. The $U$'th space of $\bigvee_{y \in Y} S^{\gamma(y)}$ is by definition given by
\[
 \bigvee_{y \in Y} \cJ(W_{\gamma(y)}, U) \sma S^{V_{\gamma(y)}},
\]
and similarly for $\bigvee_{a \in A} S^{\alpha(a)}$. Collapsing everything outside the $\epsilon$-balls in each $S^{V_y}$ to a point gives a map to $\bigvee_{a \in A} \cJ(W_{\gamma(t(a))}, U) \sma S^{V_{\gamma(t(a))}}$, and we can then use $(\tilde{t}_a)^{-1}$ to identify this with $\bigvee_{a \in A} \cJ(W_{\alpha(a)}, U) \sma S^{V_{\alpha(a)}}$.

\begin{remark}
The definition of the map
\[
 PT(\tilde{t})\colon \bigvee_{y \in Y} S^{\gamma(y)} \to \bigvee_{a \in A} S^{\alpha(a)}
\]
involved a choice of an embedding and a choice of an $\epsilon$, but any two choices give stably equivalent maps between the bouquets of spheres.
\end{remark}

The Pontryagin--Thom construction can be interpreted as a natural transformation
\[
 PT(\tilde{t})\colon S^\gamma \Rightarrow t_*^\vee S^\alpha.
\]

\begin{remark}
The natural transformation $\tilde{t}$ gives isomorphisms $V_{\gamma(t(a))}\to V_{\alpha(a)}$ and $W_{\gamma(t(a))}\to W_{\alpha(a)}$ in $\fancyI$. Taking the $G$-action into account this implies that
\[
 [\alpha(a)] = Res^{G_{t(a)}}_{G_a}[\gamma(t(a))].
\]
Hence $t\colon A \to Y$ and $\gamma \in \RO(G)(Y)$ determine $\alpha$ up to isomorphism. But note that the converse is not true: given $t\colon A \to Y$ and $\alpha \in \RO(G)(A)$ there are potentially many non-isomorphic choices of $\gamma \in \RO(G)(Y)$ that can be the target of the transfer map.
\end{remark}

Given a map $(t,\tilde{t})\colon (A,\alpha) \to (Y,\gamma)$ in $\RO(G)^T$ and a $G$-spectrum $E$, we then get a map
\[
 (t,\tilde{t})_*^\vee\colon \upis(E)(A,\alpha) \to \upis(E)(Y,\gamma)
\]
by sending a natural transformation $\phi \in [S^\alpha, \const_A(E)]_{\cB_G A}$ to the composite
\[
 S^\gamma \xToo{PT(\tilde{t})} t_*^\vee S^\alpha \xToo{t_*^\vee \phi} t_*^\vee \const_A(E) \Too \const_Y(E),
\]
where the last natural transformation is given on $y \in Y$ by the fold map $\bigvee_{t(a)=y} E \to E$. This is compatible with maps in $\RO(G)^T$, so we get a functor
\[
 \upis(E)\colon \RO(G)^T \to \Ab.
\]

If we choose isomorphisms $A \cong \coprod G/K_j$ and $Y \cong\coprod G/L_k$ in such a way that $t$ corresponds to $eK_j \mapsto eL_k$ then $(t,\tilde{t})_*^\vee$ corresponds to a combination of the usual transfer map and addition maps. Specifically, if we identify $\upis(E)(A,\alpha)$ with
\[
 \bigoplus_j [S^{\alpha(eK_j)}, E]_{K_j},
\]
then $(t, \tilde{t})$ corresponds to the composite of the standard transfer maps
\[
 \bigoplus_j[S^{\alpha(eK_j)}, E]_{K_j}\to \bigoplus_j[S^{\gamma(eL_k)}, E]_{L_k},
\]
which the Pontryagin--Thom construction produces from the isomorphism $S^{\alpha(eK_j)}\cong S^{Res^{L_k}_{K_j}\gamma(eL_k)}$, and of addition over the $eK_j$'s in the preimage of $eL_k$.

\begin{remark}
Given a map $(f,\tilde{f})\colon (Y,\gamma) \to (X,\chi)$ in $\RO(G)^T$ where $f$ is an isomorphism, the definition implies that $\tilde{f}$ is also invertible. It follows that the inverse $(f,\tilde{f})^{-1}\colon (X,\chi) \to (Y,\gamma)$ is a map in the category $\RO(G)$, and given a $G$-spectrum $E$ the two maps
\[
 (f,\tilde{f})_*^\vee,\ \big( (f,\tilde{f})^{-1} \big)^*\colon \upis(E)(Y,\gamma) \to \upis(E)(X,\chi)
\]
agree. Hence we can regard $\RO(G)^T$ as a generalization of $\RO(G)^\op$.
\end{remark}

\subsection{$\RO(G)$-graded spans and Mackey functors}\label{ss:ROGMackdefn}
We can now define a category $\RO(G)^\Mackey$ with the same objects as $\RO(G)$, where a morphism $(X, \chi) \to (Y,\gamma)$ is an equivalence class of spans
\[
 (X,\chi) \xfrom{(r,\tilde{r})} (A,\alpha) \xto{(t,\tilde{t})} (Y,\gamma).
\]
Here $(r,\tilde{r})$ is a map in $\RO(G)^R$ and $(t,\tilde{t})$ is a map in $\RO(G)^T$. The equivalence relation is generated by declaring two spans to be equivalent if there is a commutative diagram (on both the level of $G$-sets and the level of natural transformations):
\[ \xymatrix{
 & (A_1, \alpha_1) \ar[ld]_-{(r_1,\tilde{r}_1)} \ar[dd]^\cong \ar[rd]^-{(t_1,\tilde{t}_1)} & \\
(X,\chi) & & (Y,\gamma) \\
& (A_2,\alpha_2) \ar[ul]^-{(r_2,\tilde{r}_2)} \ar[ur]_-{(t_2,\tilde{t}_2)} &
} \]

Composition in $\RO(G)^\Mackey$ is given by pullback, as follows. Consider spans 
\[
 (X,\chi)\xleftarrow{(r,\tilde{r})} (A,\alpha)\xto{(t,\tilde{t})} (Y,\gamma)
\]
and 
\[
 (Y,\gamma)\xleftarrow{(r',\tilde{r}')} (B,\beta) \xto{(t',\tilde{t}')} (Z,\zeta).
\]
Their composite is the span
\[
 (X,\chi) \xleftarrow{(r,\tilde{r})} (A,\alpha) \xleftarrow{(r'',\tilde{r}'')} (D,\delta) \xto{(t'',\tilde{t}'')} (B,\beta) \xto{(t',\tilde{t'})} (Z,\zeta)
\]
where $D= A\times_Y B$ is the pullback of $A$ and $B$ in $G$-sets and $\delta\colon \cB_G D \to \fancyI^\op$ is the functor given by sending $(a,b)$ to $\beta(b)$.

Then there is an obvious map $(t'', \tilde{t}'')\colon (D,\delta) \to (B,\beta)$. To define the other map $(r'', \tilde{r}'')\colon (D,\delta) \to (A,\alpha)$ we need a map $(V_{\alpha(a)},W_{\alpha(a)}) \to (V_{\delta(a,b)}, W_{\delta(a,b)}) = (V_{\beta(b)}, W_{\beta(b)})$ in $\fancyI$. This we simply define to be the composite of the inverse of the isomorphism from $(V_{\gamma(y)}, W_{\gamma(y)})$ to $(V_{\alpha(a)}, W_{\alpha(a)})$ and the map $(V_{\gamma(y)},W_{\gamma(y)}) \to (V_{\beta(b)}, W_{\beta(b)})$.

\begin{remark}
When $(r,\tilde{r})$ and $(t',\tilde{t}')$ are the identity, we think of this composition as ``moving a restriction map past a transfer map.'' Explicitly, the composite of the transfer map $(A,\alpha)\to (Y,\gamma)$ and the restriction map $(Y,\gamma)\leftarrow (B,\beta)$ is the span
\[
 (A,\alpha) \leftarrow (D,\delta)\to (B,\beta)
\]
with $(D,\delta)$ as above.  We will use this terminology in Section \ref{s:Tambara}.
\end{remark}

With these definitions we claim that there are obvious functors from $(\RO(G)^R)^\op$ and $\RO(G)^T$ to $\RO(G)^\Mackey$ that are the identity on objects, given by inserting the identity map as one leg of the span. It suffices to observe that the composition laws are compatible. For $\RO(G)^R$ this is clear: if $(t,\tilde{t})\colon (A,\alpha) \to (Y,\gamma)$ is the identity then $(D,\delta)=(B,\beta)$ and $(r'',\tilde{r}'')=(r',\tilde{r}')$.

For $\RO(G)^T$ this is somewhat less obvious. Using the composition law in $\RO(G)^T$ we get
\[
 (A,\alpha) \xfrom{(=,=)} (A,\alpha) \xto{(t' \circ t, t^* \tilde{t}' \circ \tilde{t})} (Z,\zeta)
\]
while using the composition law in $\RO(G)^\Mackey$ we get
\[
 (A,\alpha) \xfrom{(=,\tilde{t}^{-1})} (A, \gamma \circ t) \xto{(t' \circ t, t^* \tilde{t}')} (Z,\zeta).
\]
These define equivalent spans because the isomorphism $(A,\alpha)\to (A,\gamma\circ t)$ given by $(\id, \tilde{t})$ fits into the following commutative diagram:
\[ \xymatrix{
 & (A, \alpha) \ar[ld]_{(\id,\id)} \ar[dd]_{(\id,\tilde{t})} \ar[rd]^{(t'\circ t, t^*\tilde{t}'\circ \tilde{t})} & \\
(A,\alpha) & & (Z,\zeta) \\
& (A,\gamma\circ t) \ar[ul]^{(\id,\tilde{t}^{-1})} \ar[ur]_{(t'\circ t, t^*\tilde{t}')} &
} \]
Thus we have a well-defined functor $\RO(G)^T\to \RO(G)^\Mackey$ given by including $\RO(G)^T$ as spans where the left leg is the identity.

It is clear that composition is unital and associative, so that we have a well defined category $\RO(G)^\Mackey$.  An $\RO(G)$-graded Mackey functor is then a functor out of this category.
\begin{defn}
An $\RO(G)$-graded Mackey functor is a finite product-preserving functor 
\[M\colon \RO(G)^\Mackey \to \Ab.\]
\end{defn}
\begin{remark}
The product in $\RO(G)^\Mackey$ is given by disjoint union of $G$ sets, and the condition that a graded Mackey functor be product preserving enforces the condition that $M(X\coprod Y,\chi\coprod \gamma)\cong M(X,\chi)\oplus M(Y,\gamma)$.  
\end{remark}

Now we can restate Theorem \ref{t:Mackey}.

\begin{thm} \label{t:Mackey_restated}
Let $E$ be a $G$-spectrum. Then $E$ determines an $\RO(G)$-graded Mackey functor
\[
 \upis(E)\colon \RO(G)^\Mackey \to \Ab
\]
sending $(X,\chi)$ to $\upis(E)(X,\chi)$.
\end{thm}

\begin{proof}  Note that by definition, $\upis(E)$ takes disjoint union in $\RO(G)$ to direct sum of abelian groups, and is thus finite product-preserving.
Since we have functors from $\RO(G)^R$ and $\RO(G)^T$ into $\RO(G)^\Mackey$, the only thing left to prove is that given a diagram
\[ \xymatrix{
 & (D,\delta) \ar[ld]_-{(r', \tilde{r}')} \ar[rd]^-{(t',\tilde{t}')} & \\
(A,\alpha) \ar[rd]_-{(t,\tilde{t})} & & (B,\beta) \ar[ld]^-{(r, \tilde{r})} \\
& (Y,\gamma)
} \]
defining the composite $(r,\tilde{r}) \circ (t,\tilde{t})$ in $\RO(G)^\Mackey$ the two maps
\[
 (r, \tilde{r})^* \circ (t,\tilde{t})_*^\vee, (t',\tilde{t}')_*^\vee \circ (r', \tilde{r}')^*\colon \upis(E)(A,\alpha) \to \upis(E)(B,\beta)
\]
agree.

The first map is given by sending $\phi\colon S^\alpha \Rightarrow \const_A(E)$ to the composite
\begin{multline*}  \hspace{1cm}S^\beta  \xToo{\tilde{r}}  r^* S^\gamma \xToo{r^* PT(\tilde{t})} r^* t_*^\vee S^\alpha 
   \xToo{r^*t_*^\vee \phi} \\ r^* t_*^\vee \const_A(E) 
 \Too  r^* \const_Y(E) = \const_B(E)\hspace{1cm}
\end{multline*} 
while the second map sends $\phi$ to
\begin{multline*}
 S^\beta\xToo{PT(\tilde{t}')}  (t')_*^\vee S^\delta \xToo{(t')_*^\vee (r')^*}  (t')_*^\vee (r')^* S^\alpha \xToo{(t')_*^\vee (r')^* \phi}  \\(t')_*^\vee (r')^* \const_A(E) = (t')_*^\vee \const_D(E) \To \const_B(E).
\end{multline*}
%
%
The result then follows from three observations. First, the functors $r^* t_*^\vee S^\alpha$ and $(t')_*^\vee (r')^* S^\alpha$ are naturally isomorphic, with both given by $b \mapsto \bigvee_{t(a)=r(b)} S^{\alpha(a)}$, and the two maps from $S^\beta$ are homotopic.

Second, the functors $r^* t_*^\vee \const_A(E)$ and $(t')_*^\vee (r')^* \const_A(E)$ are naturally isomorphic, and the two maps from $r^* t_*^\vee S^\alpha \cong (t')_*^\vee (r')^* S^\alpha$ agree.

And third, the two maps from $r^* t_*^\vee \const_A(E) \cong (t')_*^\vee (r')^* \const_A(E)$ to $\const_B(E)$ agree.
\end{proof}

\begin{remark}
The standard double coset formula for Mackey functors follows from the proof of Theorem \ref{t:Mackey_restated} just as in the standard definition of Mackey functors in terms of functors on the Burnside category.  Concretely, let $G/H$ and $G/K$ be $G$-orbits.  The usual formulation of the double coset formula is the equality
\[
 Res^{G}_{K} Tr^{G}_H=\sum_{g} Tr_{g^\inv Hg\cap K}^K c_g Res^H_{H\cap{}^gK}.
\]
Given maps  $(t,\tilde{t})\colon (G/H, \alpha)\to (G/G,\gamma)$  in $\RO(G)^T$ and $(r,\tilde{r})\colon(G/K,\beta)\to(G/G,\gamma)$ in $\RO(G)^R$, we can derive this equality by considering their composite in $\RO(G)^\Mackey$.  We calculate this composite via the pullback square
\[ \xymatrix{
 & (D,\delta) \ar[ld]_-{(r', \tilde{r}')} \ar[rd]^-{(t',\tilde{t}')} & \\
(G/H,\alpha) \ar[rd]_-{(t,\tilde{t})} & & (G/K,\beta) \ar[ld]^-{(r, \tilde{r})} \\
& (G/G,\gamma)
} \]
where $D= G/H\times_{G/G} G/K$ and $\delta=\beta\circ \pi_2$. 
By definition, the composite $(r,\tilde{r})^*\circ (t,\tilde{t})_*$ is
\[
 Res^{G}_{K} Tr^{G}_H.
\]
As in the proof of Theorem \ref{t:Mackey_restated}, this is equal to the composite $(t',\tilde{t}')_*\circ(r',\tilde{r}')^*$.  We identify this latter composite as the right hand side of the double coset formula as follows. 
 The pullback $D$ decomposes into $G$-orbits as
\[
 D= \coprod_{g\in G} G/(H\cap{}^{g}\!K)
\]
where ${}^g\!K$ denotes the conjugate $gKg^\inv$. Hence the composite $(t',\tilde{t}')_*\circ(r',\tilde{r}')^*$ is
\[
 \sum_{g} Tr_{g^\inv Hg\cap K}^K c_g Res^H_{H\cap{}^gK}.
\]
We can think of the conjugation $c_g$ as picking a different basepoint for each orbit---it is a different choice of the equivalence of categories in  Section  \ref{subsect:equivalenceofcat}.
\end{remark}

\section{Graded bispans and graded Tambara functors} \label{s:Tambara}
Recall that an ordinary Tambara functor consists of an abelian group $T(X)$ for each finite $G$-set $X$, and a morphism $T(X) \to T(Y)$ for each diagram
\[
 X \xfrom{r} A \xto{n} B \xto{t} Y.
\]
The map $n\colon A \to B$ determines a \emph{norm map} $n_*^\sma\colon T(A) \to T(B)$.  In this section, we define the $\RO(G)$-graded version of Tambara functors.

\subsection{The category of norm maps}
We define a category $\RO(G)^N$ with the same objects as $\RO(G)$ as follows.

\begin{defn}
Let $(A,\alpha)$ and $(B,\beta)$ be objects in $\RO(G)$. A map in $\RO(G)^N$ from $(A,\alpha)$ to $(B,\beta)$ is a pair $(n,\tilde{n})$ where $n\colon A \to B$ is a $G$-map and
\[
 \tilde{n}\colon \beta \Rightarrow n_*^\oplus \alpha 
\]
is a natural transformation of functors from $\cB_G B$ to $\fancyI^\op$. Here $n_*^\oplus \alpha\colon B \to \fancyI^\op$ is given by $n_*^\oplus \alpha (b) = \bigoplus_{n(a)=b} \alpha(a)$.

Composition in $\RO(G)^N$ is defined as follows:  Let $(m,\tilde{m})\colon (A,\alpha)\to (B,\beta)$ and $(n,\tilde{n})\colon (B,\beta)\to (C,\zeta)$ be maps in $\RO(G)^N$.  Then their composite is the pair $(n\circ m,\widetilde{n\circ m})$ where the natural transformation $\widetilde{n\circ m}\colon \zeta\to (n\circ m)^\oplus_*\alpha$ is the composite
\[ \zeta \xRightarrow{\tilde{n}} n^\oplus_*\beta \xRightarrow{n^\oplus_*\tilde{m}} n^\oplus_*m^\oplus_*\alpha \Rightarrow (n\circ m)^\oplus_*\alpha\]
where the final map is the natural isomorphism $n^\oplus_*\circ m^\oplus_*\approx (n\circ m)^\oplus_*$ arising from the symmetric monoidal structure on $\fancyI^\op$.  By \cite[Proposition A.29]{HHR}, such natural isomorphisms are compatible with composition.  That is, given composible maps of finite $G$-sets $l, m$ and $n$, the following diagram of natural isomorphisms commutes:
\[
\xymatrix{ n^\oplus_*\circ m^\oplus_*\circ l^\oplus_* \ar[r]\ar[d]& n^\oplus\circ (m\circ l)^\oplus_*\ar[d]\\
(n\circ m)^\oplus_*\circ l^\oplus_*\ar[r] &(n\circ m\circ l)^\oplus_*
}\]
From this compatibility, we deduce that composition in $\RO(G)^N$ is associative.
\end{defn}

\begin{remark}
Note that there are no maps from $(A,\alpha)$ to $(B,\beta)$ unless $[\beta(b)] = \bigoplus_{n(a)=b} [\alpha(a)]$ as virtual vector spaces. In the case when $A=G/H$ and $B=G/K$, with $H \leq K$ and $n\colon G/H \to G/K$ the canonical map, it follows that $[\beta(eK)] = N_H^K[\alpha(eH)]$ in the ordinary representation ring $RO(K)$.
\end{remark}

Because $\cF(-)$ is a symmetric monoidal functor from $\fancyI^\op$ to $\Spec$ it follows that $S^{n_*^\oplus \alpha}$ is isomorphic to $n_*^\sma S^\alpha$ and that a natural transformation $\beta \Rightarrow n_*^\oplus \alpha$ in $\fancyI^\op$ induces a natural transformation $S^\beta \Rightarrow n_*^\sma S^\alpha$ of functors from $\cB_G B$ to $\Spec$.

\begin{remark}
Given a map $(f, \tilde{f})\colon (Y,\gamma) \to (X,\chi)$ in $\RO(G)^N$ where $f$ is an isomorphism, the natural transformation $\tilde{f}$ takes the form $S^{\chi} \To f_*^\sma S^\gamma = (f^{-1})^* S^\gamma$, so we see that $(f,\tilde{f})$ gives the same data as a map $(X,\chi) \to (Y,\gamma)$ in the category $\RO(G)$. Hence $\RO(G)^N$ is also a natural generalization of $\RO(G)^\op$.
\end{remark}

\subsection{$\RO(G)$-graded bispans and Tambara functors} 
We now define the category $\RO(G)^{\Tambara}$ that is the domain of $\RO(G)$-graded Tambara functors. Again, the objects of $\RO(G)^{\Tambara}$ are the same as those of $\RO(G)$.
\begin{defn}
Let $(X,\chi)$ and $(Y,\gamma)$ be objects of $\RO(G)$.  A map in $\RO(G)^{\Tambara}$ from $(X,\chi)$ to $(Y,\gamma)$ is an equivalence class of bispans
\[ (X,\chi) \xleftarrow{(r,\tilde{r})} (A,\alpha) \xto{(n,\tilde{n})} (B,\beta) \xto{(t,\tilde{t})} (Y,\gamma).\]
Here $(r,\tilde{r})$ is a map in $\RO(G)^R$, $(n, \tilde{n})$ is a map in $\RO(G)^N$ and $(t,\tilde{t})$ is a map in $\RO(G)^T$.  The equivalence relation is generated by declaring two bispans to be equivalent if there is a commutative diagram on the level of $G$-sets and natural transformations
\[
\xymatrix{
&(A_1,\alpha_1) \ar[r]\ar[dl]\ar[dd]^{\cong} &(B_1, \beta_1)\ar[dr]\ar[dd]^{\cong} \\
(X,\chi) & & & (Y,\gamma)\\
& (A_2,\alpha_2)\ar[r]\ar[ul] & (B_2,\beta_2)\ar[ur]
}
\]
\end{defn}

Composition in $\RO(G)^{\Tambara}$ is given by a generalization of the composition formula in \cite{Ta93}. If 
\[ (X,\chi) \xleftarrow{(r,\tilde{r})} (A,\alpha) \xto{(n,\tilde{n})} (B,\beta)\xto{(t,\tilde{t})} (Y,\gamma).\]
and 
\[ (Y,\gamma) \xleftarrow{} (C,\xi)\xto{} (D,\delta) \xto{} (Z,\zeta)\]
are bispans in $\RO(G)^\Tambara$, then their composite is the bispan
\[ (X,\chi) \xleftarrow{} (A'',\alpha'') \xto{} (D',\delta') \xto{} (Z,\zeta)\]
defined by the following diagram:
\[\xymatrix{
 (A'',\alpha'')\ar[r]\ar[d]\ar@{}[dr]|*+[Fo]{1} & (C',\xi')\ar[r]\ar[d]\ar@{}[dr]|*+[Fo]{2} &(D',\delta')\ar[dr] \\
(A',\alpha')\ar[r]\ar[d]\ar@{}[dr]|*+[Fo]{3}& (B',\beta')\ar[r]\ar[d]\ar@{}[dr]|*+[Fo]{4} & (C,\xi)\ar[r]\ar[d] & (D,\delta) \ar[r]& (Z,\zeta)\\
(A,\alpha)\ar[r]\ar[d]& (B,\beta)\ar[r] & (Y,\gamma)\\
(X,\chi)
}\]
At the level of $G$-sets, the squares 1, 3, and 4 in this diagram are pullback squares  and the pentagon 2 is an ``exponential diagram,'' as defined in \cite[Section 1]{Ta93} or below in Section \ref{sec:moving-norm-map-transfer}. The gradings are determined as follows.  We use $pr$ to denote projection from a pullback to the component specified in the subscript.  Square 4 is given by moving a restriction past a transfer map, as described in Section \ref{ss:ROGMackdefn} above, so that $\beta'$ is the functor $pr_C^*\xi$.  Squares 1 and 3 are both instances of moving a restriction past a norm map, as described in Section \ref{sec:moving-restr-map-norm} below, so that $\alpha'=pr_A^*\alpha$ and $\alpha''=pr_{A'}^*\alpha$.  The definition of the gradings in the exponential diagram is made explicit in Section \ref{sec:moving-norm-map-transfer} below.

Indeed, the descriptions of moving restriction maps past norm maps and norm maps past transfer maps given in the next two sections prove the following proposition.  Compare with the non-graded version in \cite[Proposition 7.3]{Ta93}.
\begin{prop} The category $\RO(G)^{\Tambara}$ is generated by maps of the form
\begin{enumerate}
\item \emph{(Restriction type)} $(X,\chi) \leftarrow (A,\alpha) = (A,\alpha) = (A,\alpha)$ 
\item \emph{(Norm type)} $(A,\alpha) =  (A,\alpha) \to (B,\beta) =(B,\beta)$ 
\item \emph{(Transfer type)} $(B,\beta) = (B,\beta) =(B,\beta)\to (Y,\gamma)$
\end{enumerate}
under relations given by the above diagram.
\end{prop}

\subsection{Moving a restriction map past a norm map}\label{sec:moving-restr-map-norm}
First we consider the following. Given a norm map $(n, \tilde{n})\colon (A,\alpha) \to (B,\beta)$ in $\RO(G)^N$ and a restriction map $(r,\tilde{r})\colon (X,\chi) \to (B,\beta)$ in $\RO(G)^R$ we define the composite  $(r,\tilde{r})\circ(n,\tilde{n})$  to be the bispan
\[
 (A,\alpha) \xfrom{(r',\tilde{r}')} (D,\delta) \xto{(n',\tilde{n}')} (X,\chi) \xto{=} (X,\chi)
\]
where $D$ is the pullback $D=A \times_B X$ and $\delta\colon \cB_G D \to \fancyI^\op$ is given as follows. On an object $d=(a,x) \in D$ we set $\delta(a,x) = \alpha(a)$, and on a morphism $g\colon (a,x) \to (ga, gx)$ we set $\delta(g)$ to be $\alpha(g)\colon \alpha(a) \to \alpha(ga)$.

The natural transformation $\tilde{r}'\colon \delta \Rightarrow (r')^* \alpha$ is the identity on objects, while the natural transformation $\tilde{n}'\colon \chi \Rightarrow (n')^\oplus_* \delta$ is given on objects by the composite
\[
 \chi(x) \xto{\tilde{r}} \beta(r(x)) \xto{\tilde{n}} \bigoplus_{n(a)=r(x)} \alpha(a) = \bigoplus_{d \in D, n'(d)=x} \delta(d).
\]
It follows that $(r',\tilde{r}')$ is in $\RO(G)^R$ and $(n',\tilde{n}')$ is in $\RO(G)^N$.

\begin{remark}
Note that here, as in the definition of the composite of a transfer and a restriction map, the composite uses the pullback $A\times_B X$, but the transfer/restriction case uses a different grading.  In the present case, we project to $A$, the source of the norm map; in the transfer/restriction case, we project to the source of the restriction map.  The difference reflects the different requirements on the gradings in these two cases.
\end{remark}

\subsection{Moving a norm map past a transfer map}\label{sec:moving-norm-map-transfer}
This is the most complicated composition rule in $\RO(G)^\Tambara$. Given $(t,\tilde{t})\colon (C,\xi) \to (A,\alpha)$ in $\RO(G)^T$ and $(n,\tilde{n})\colon (A,\alpha) \to  (B,\beta)$ in $\RO(G)^N$ we define the composite $(n,\tilde{n})\circ(t,\tilde{t})$ to be the bispan
\[
 (C,\xi) \xfrom{(r,\tilde{r})} (\otherE,\epsilon) \xto{(n',\tilde{n}')} (D,\delta) \xto{(t',\tilde{t}')} (B,\beta)
\]
defined in terms of the ``exponential diagram''
\[\xymatrix{ (\otherE,\epsilon)\ar[d]_{(r,\tilde{r})}\ar[r]^-{(n',\tilde{n}')} & (D,\delta)\ar[dr]^-{(t',\tilde{t}')}\\
(C,\xi)\ar[r]^-{(t,\tilde{t})} & (A,\alpha)\ar[r]^-{(n,\tilde{n})} & (B,\beta)}
\]
We define
\begin{eqnarray*}
 D & = & \{(b,s) \mid b \in B, \,\, s\colon n^{-1}(b) \to C \textnormal{ is a section of $t$}\} \\
 \delta(b,s) & = & \beta(b) \\
 \otherE & = & \{(a,b,s) \mid (b,s) \in D, \,\, a \in n^{-1}(b) \} \\
 \epsilon(a,b,s) & = & \xi(s(a))
\end{eqnarray*}
The maps $r\colon \otherE \to C$, $n'\colon \otherE \to D$ and $t'\colon D \to B$ are the obvious ones, and the natural transformations are defined as follows.

We define $\tilde{r}\colon \epsilon \To r^* \xi$ and $\tilde{t}'\colon \delta \Rightarrow \beta \circ t'$ to be the identity on objects. The natural transformation $\tilde{n}'\colon \delta \Rightarrow (n')_*^\oplus \epsilon$ is given on objects by maps
\[
 \delta(b,s) = \beta(b) \to \bigoplus_{n'(a,b,s)=(b,s)} \epsilon(a,b,s) = \bigoplus_{n(a)=b} \xi(s(a))
\]
which we now define.

The natural transformation $\tilde{n}$ gives a map $\beta(b) \to \bigoplus_{n(a)=b} \alpha(a)$, and the natural transformation $\tilde{t}$ gives an isomorphism $\xi(s(a)) \to \alpha(a)$, so we can define the $(b,s)$'th component of $\tilde{n}'$ to be the composite
\[
 \beta(b) \xto{\tilde{n}_b} \bigoplus_{n(a)=b} \alpha(a) \xto{\oplus \tilde{t}_a^{-1}} \bigoplus_{n(a)=b} \xi(s(a)).
\]
Notice that the set of $a \in A$ with $n(a)=b$ is in bijection with the set of $(a,b,s) \in E$ with $n'(a,b,s)=(b,s)$.

The behavior on morphisms should be clear from this description. and it follows from the definitions that indeed $(r,\tilde{r})$ is in $\RO(G)^R$, $(n',\tilde{n}')$ is in $\RO(G)^N$ and $(t',\tilde{t}')$ is in $\RO(G)^T$.

\subsection{Graded Tambara functors}
Now that we have a category $\RO(G)^\Tambara$, we can define $\RO(G)$-graded Tambara functors.  This requires more care than defining $\RO(G)$-graded Mackey functors: in particular, we cannot simply consider functors from $\RO(G)^\Tambara$ to abelian groups because norm maps are not maps of abelian groups.  Fundamentally, this is because, for a ring $R$, the multiplication map $R\times R\to R$ is not a homomorphism of abelian groups, but rather just a map of sets.  Since the norm maps are a generalization of multiplication, we cannot require them to be maps in the category of abelian groups either.

The solution is to consider product-preserving functors from $T\colon \RO(G)^\Tambara$ to the category of sets and then to observe that the image $T(X,\chi)$ of any object in $\RO(G)^\Tambara$ inherits a commutative monoid structure.  In analogy with \cite{Ta93}, we could call such functors ``graded semi-Tambara functors;''
because they are graded semi-ring version of Tambara functors.  In the nongraded case, Tambara functors are simply semi-Tambara functors where each commutative monoid $T(X)$ is in fact an abelian group, which motivates Definition \ref{defngradedTambarafunctor}.

\begin{lemma} Let $T\colon \RO(G)^\Tambara \to \Set$ be a finite product preserving functor.  For every finite $G$-set $X$ and each grading $\chi\colon B_GX\to \fancyI^\op$, the fold map $\nabla\colon X\coprod X\to X$ induces a commutative monoid structure on $T(X,\chi).$
\end{lemma}
\begin{proof}
Notice that for any $G$-sets $X$ and $Y$, the translation category $B_G(X\coprod Y)$ is the coproduct of categories $B_GX\coprod B_GY$.  The grading $\chi$ thus induces a grading $\chi\coprod\chi$ on $X\coprod X$.  Let $\nabla\colon X\coprod X\to X$ be the fold map and observe that $\nabla^*\chi=\chi \coprod \chi$.  We therefore have a map
\[ (X\coprod X,\chi\coprod\chi)\xto{(\nabla, \id)} (X,\chi)\]
in $\RO(G)^T$, which we will continue to call the ``fold map.''  This fold map then induces a map in $\RO(G)^\Tambara$
\[ (X,\chi)\coprod (X,\chi) \leftarrow (X,\chi)\coprod (X,\chi) \to (X,\chi)\coprod (X,\chi)\xto{(\nabla,\id)}(X,\chi)\]
which gives an addition-type map on $(X,\chi)$.  
The inclusion of the empty set $i\colon \emptyset \to X$ induces a unit map 
\[(\emptyset,\emptyset)\leftarrow (\emptyset,\emptyset)\to (\emptyset,\emptyset)\xto{(i,\id)} (X,\chi)\]
and it is straightforward to check that these maps give $(X,\chi)$ the structure of a commutative monoid object in $\RO(G)^\Tambara$.   Hence the image of $(X,\chi)$ under the product-preserving functor $T$ is a commutative monoid object in $\Set$, i.e.,\,a commutative monoid.
\end{proof}

\begin{defn}\label{defngradedTambarafunctor}
An $\RO(G)$-graded Tambara functor is a finite product-preserving functor
\[ T\colon \RO(G)^\Tambara \to \Set\]
such that the inherited commutative monoid structure on each $(X,\chi)$ in $\RO(G)^\Tambara$ is in fact an abelian group.
\end{defn}

\subsection{Norm multiplication}
Recall that
\[
 \constant_A(E)\colon \cB_G A \to \Spec
\]
is the functor that sends any $a \in \cB_G A$ to $E$, and that $\upis(E)(A,\alpha)$ is canonically isomorphic to the homotopy classes of natural transformations from $S^\alpha$ to $\constant_A(E)$, taken with respect to the model structure on the category $Fun(\cB_G A,\Spec)$. Given a transformation $\phi\colon S^\alpha \Rightarrow \constant_A(E)$, the map $(A,\alpha)\xto{(n,\tilde{n})} (B,\beta)$ gives a composite
\[
 S^\beta \Rightarrow n_*^\sma S^\alpha \Rightarrow n_*^\sma \constant_A(E)
\]
of natural transformations. In order to produce an element of $\upis(E)(B,\beta)$ we would like a natural transformation whose codomain is $\constant_B(E)$. This motivates the following definition.

\begin{defn} \label{d:normmult}
A \emph{norm multiplication} on a $G$-spectrum $E$ is, for each $n\colon A \to B$,  a morphism 
\[
 \mu_A^B\colon n_*^\sma \constant_A(E) \Rightarrow \constant_B(E)
\]
in the homotopy category of $Fun(\cB_GB,\Spec)$ as described in \S \ref{subsect:relatingtoHhtpygrps}.  These morphisms must  satisfy the following properties.
\begin{enumerate}
 \item Given $n\colon A \to B$ and $n'\colon B \to C$ the composite $\mu_B^C \circ (n')_*^\sma(\mu_A^B)$ is equivalent to $\mu_A^C$ in the homotopy category of $Fun(\cB_GC,\Spec)$.
 \item The natural transformations are stable under pullback. 
\end{enumerate}
\end{defn}

Note that we are only asking for maps in the homotopy category, which makes sense because $n_*^\sma$ is a left derived functor. For example, if $G=\{e\}$, the homotopy category of $Fun(\cB_{\{e\}} X,\Spec)$ is the category $Fun(\cB_{\{e\}} X,\hoSpec)$ and so we are asking for a homotopy associative and homotopy commutative multiplication.

The second condition in Definition \ref{d:normmult} says that if
\[ \xymatrix{
 A' \ar[r]^-{n'} \ar[d]_i & B' \ar[d]^j \\
 A \ar[r]^-n & B
} \]
is a pullback diagram then $\mu_{A'}^{B'}$ is the natural transformation 
\[
 (n')_*^\sma \constant_{A'}(E) = j^* n_*^\sma \constant_A(E) \stackrel{j^* \mu_A^B}{\Longrightarrow} j^* \constant_B(E) = \constant_{B'}(E).
\]

\begin{example}\label{ex:HHRnorm}
Suppose $H \leq K$ and $n\colon G/H \to G/K$ is the canonical map. Then $n_*^\sma \constant_{G/H}(E)$ evaluated at $eK$ is the Hill--Hopkins--Ravenel norm $N_H^K E$, so in this case the natural transformation amounts to a map $N_H^K E \to E$.
\end{example}

\begin{example}\label{ex:usualmult}
Suppose we start with the map $n\colon G/G \coprod G/G \to G/G$. Then $n_*^\sma \constant_{G/G \coprod G/G}(E)$ evaluated at $G/G$ is $E \sma E$, so in this case the natural transformation amounts to a $G$-map $E \sma E \to E$.

Applying Property (1) of Definition \ref{d:normmult} to the precomposition of $n$ and the twist map $G/G \coprod G/G \xto{\tau} G/G \coprod G/G$  shows that the $G$-map $E \sma E \to E$ is invariant under the twist map. Because this all takes place up to homotopy this means that $E$ has a homotopy commutative multiplication. Homotopy associativity and unitality follows by considering similar diagrams.
\end{example}

\begin{remark}
By decomposing $G$-sets into orbits, we can write any map $n\colon A\to B$ of $G$-sets as a composite of maps of the type of Example~\ref{ex:HHRnorm} and Example \ref{ex:usualmult}.  Thus the existence of norm multiplications for a spectrum $E$ is equivalent to the existence of compatible multiplications $N^K_H(i^*_H E)\to i^*_K E$ for all subgroups $H, K\leq G$ and a usual multiplication map $E\sma E\to E$ in the homotopy category of orthogonal $G$-spectra. 
\end{remark}

\begin{example}
Consider the pullback diagram
\[ \xymatrix{
 G/e \coprod G/e \ar[r] \ar[d] & G/e \ar[d] \\
 G/G \coprod G/G \ar[r] & G/G
} \]
The norm multiplication corresponding to the bottom map amounts to a $G$-map $E \sma E \to E$ and the norm multiplication corresponding to the top map amounts to a non-equivariant map $E \sma E \to E$. The second condition in the definition of a norm multiplication then implies that the non-equivariant map $E \sma E \to E$ is the underlying map of the $G$-equivariant map $E \sma E \to E$.

Another way of conceptualizing this compatibility is as follows.  The norm multiplications induced by the pullback diagram above yield the following diagram of maps of spectra:
\[\xymatrix{
N^G_e(i^*_e E)\sma N^G_e(i^*_e E) \ar[d]\ar[r] & N^G_e(i^*_e E)\ar[d]\\
E\sma E \ar[r] & E}
\]
The right vertical map is the Hill--Hopkins--Ravenel norm and the left vertical map is the smash product of the norm with itself, whereas the horizontal maps are ordinary multiplications.  Hence this diagram tells us that the norm $N^G_e(i^*_e E)\to E$ is a map of commutative monoids.\footnote{Thanks to Mike Hill for pointing out this alternate interpretation.}
\end{example}

\begin{example}  Any commutative orthogonal $G$-ring spectrum $E$ has norm multiplication. This follows from the fact that in the category of commutative orthogonal ring spectra the norm, as constructed by Hill--Hopkins--Ravenel, is left adjoint to the restriction functor \cite[Corollary A.56]{HHR}.  The counit of this adjunction is a map 
\[
 N^G_H i^*_H E\to E
\]
which produces the norm multiplication.  Indeed, this adjunction is a special case of the adjunction of Proposition \ref{p:smashad} (cf. \cite[\S A.3.5]{HHR}).  To be explicit, let $n\colon A\to B$ be a map of finite $G$-sets, and consider the adjunction $n^\sma_* \dashv n^*$ of Proposition~\ref{p:smashad}.  The identity map $\const_A(E) \to n^*\const_B(E)$ has adjunct the desired norm multiplication $n^\sma_*\const_A E\to \const_B E$.  Compatibility condition (1) is a standard property of adjunctions. Condition (2) follows from inspection of the definitions of the counit in the adjunction of Proposition \ref{p:smashad}.
\end{example}

\begin{remark} \label{r:Ninfinity}
As discussed in \cite{BlHi13}, there are multiple generalizations of $E_\infty$ operads to the $G$-equivariant context, known as $N_\infty$ operads. Given an $N_\infty$ operad $\cO$, there is a concomitant notion of an $\cO$-admissible $H$-set for each $H \leq G$.

Given a collection of admissible $H$-sets for each $H \leq G$ satisfying the conditions spelled out in \cite[Section 4]{BlHi13}, we say that $n\colon A \to B$ is admissible if each $n^{-1}(b)$ is an admissible $G_b$-set. If we only use admissible norm maps we arrive at the notion of an incomplete $\RO(G)$-graded Tambara functor, and we could modify Definition \ref{d:normmult} by only asking for admissible norm multiplication.

Then any $\cO$-algebra $E$ has $\cO$-admissible norm multiplication, and given a $G$-spectrum $E$ with admissible norm multiplication the obvious generalization of Theorem \ref{t:Tambara_restated} below still holds. We have avoided the additional generality in an attempt at keeping the paper readable.
\end{remark}

Now we can prove Theorem \ref{t:Tambara}, which we restate here for convenience.

\begin{thm} \label{t:Tambara_restated}
Let $E$ be an orthogonal $G$-spectrum with a norm multiplication. Then $E$ determines an $\RO(G)$-graded Tambara functor
\[
 \upis(E)\colon \RO(G)^\Tambara \to \Ab.
\]
sending $(X,\chi)$ to $\upis(E)(X,\chi)$.
\end{thm}

\begin{proof}
There are only two things left to prove. First, we need to prove that our definition of $\upis(E)$ respects composition of norm maps and restriction maps. Suppose we have a norm map $(n,\tilde{n})\colon (A,\alpha)\to (B,\beta)$ and a restriction map $(r,\tilde{r})\colon (X,\chi)\to (B,\beta)$ as in Section \ref{sec:moving-restr-map-norm}.  We need to prove that the two composites
\[
 (r,\tilde{r})^* \circ (n,\tilde{n})_*^\sma,\ (n',\tilde{n}')_*^\sma \circ (r', \tilde{r}')^*\colon \upis(E)(A,\alpha) \to \upis(E)(X,\chi)
\]
agree.

The first composite is given by sending a natural transformation $\phi\colon S^\alpha \To \const_A(E)$ to the composite
\begin{multline*}
\hspace{1cm} S^\chi \xToo{\tilde{r}}  r^* S^\beta \xToo{r^* n_*^\sma}  r^* n_*^\sma S^\alpha 
 \xToo{r^* n_*^\sma \phi} \\ r^* n_*^\sma \const_A(E)  
 \xToo{r^* \mu_A^B}  r^* \const_B(E)= \const_X(E)\hspace{1cm}
\end{multline*}
while the other sends $\phi$ to
\begin{multline*}
 S^\chi \xToo{\tilde{n}'}  (n')_*^\sma S^\delta \xToo{(n')_*^\sma \tilde{r}'}  (n')_*^\sma (r')^* S^\alpha \xToo{(n')_*^\sma (r')^* \phi} \\ (n')_*^\sma (r')^* \const_A(E)  = (n')_*^\sma \const_D(E) \xToo{\mu_D^X}  \const_X(E)
\end{multline*}

It follows that the composites are equal from three observations. First, the functors $r^* n_*^\sma S^\alpha$ and $(n')_*^\sma (r')^* S^\alpha$ are naturally isomorphic, with both given by $x \mapsto \sma_{n(a)=r(x)} S^{\alpha(a)}$, and the two maps from $S^\chi$ are isomorphic.

Second, the functors $r^* n_*^\sma \const_A(E)$ and $(n')_*^\sma (r')^* \const_A(E)$ are naturally isomorphic, and the two maps from $r^* n_*^\sma S^\alpha \cong (n')_*^\sma (r')^* S^\alpha$ are homotopic.

And third, the two maps from $r^* n_*^\sma \const_A(E) \cong (n')_*^\sma (r')^* \const_A(E)$ to $\const_X(E)$ agree. This follows from the condition that the norm multiplication maps are stable under pullback.

%


The final thing to prove is that given an exponentiation diagram as in Section \ref{sec:moving-norm-map-transfer} the two maps
\[
 (n,\tilde{n})_*^\sma \circ (t,\tilde{t})_*^\vee,\  (t',\tilde{t}')_*^\vee \circ (n',\tilde{n}')_*^\sma \circ (r,\tilde{r})^*\colon \upis(E)(C,\xi) \to \upis(E)(B,\beta)
\]
agree.

The first composite is given by sending a natural transformation $\phi\colon S^\xi \to \const_C(E)$ to the composite
\begin{multline*}
 \hspace{1cm}S^\beta \xToo{\tilde{n}}  n_*^\sma S^\alpha \xToo{n_*^\sma PT(\tilde{t})}  n_*^\sma t_*^\vee S^\xi  \xToo{n_*^\sma t_*^\vee \phi} \\ n_*^\sma t_*^\vee \const_C(E) 
 \To  n_* \const_A(E) \xToo{\mu_A^B} \const_X(E)\hspace{1cm}
\end{multline*}
while the other sends $\phi$ to
\begin{multline*}
 S^\beta \xToo{PT(\tilde{t}')}  (t')_*^\vee S^\delta \xToo{(t')_*^\vee \tilde{n}'}  (t')_*^\vee (n')_*^\sma S^\epsilon \xToo{(t')_*^\vee (n')_*^\sma \tilde{r}} \\ (t')_*^\vee (n')_*^\sma r^* S^\xi  \xToo{(t')_*^\vee (n')_*^\sma r^* \phi}  (t')_*^\vee (n')_*^\sma r^* \const_C(E)\\ = (t')_*^\vee (n')_*^\sma \const_{\otherE}(E) \xToo{(t')_*^\vee \mu_{\otherE}^D}  (t')_* \const_D(E) \To \const_B(E)
\end{multline*}
Again it follows that these two composites are equal from three observations. First, the functors $n_*^\sma t_*^\vee S^\xi$ and $(t')_*^\vee (n')_*^\sma r^* S^\xi$ are naturally isomorphic, because the value on $b$ is given by
\[
 \bigwedge_{n(a)=b} \Big( \bigvee_{t(c)=a} S^{\xi(c)} \Big) \cong \bigvee_{t'(d)=b} \Big( \bigwedge_{n'(e)=d} S^{\xi(r(e))} \Big).
\]
Moreover, the two maps from $S^\beta$ are isomorphic.

Second, the functors $n_*^\sma t_*^\vee \const_C(E)$ and $(t')_*^\vee (n')_*^\sma r^* \const_C(E)$ are naturally isomorphic, and the two maps from $n_*^\sma t_*^\vee S^\xi \cong (t')_*^\vee (n')_*^\sma r^* S^\xi$ are homotopic.

And third, the two maps from $n_*^\sma t_*^\vee \const_C(E) \cong (t')_*^\vee (n')_*^\sma r^* \const_C(E)$ agree. This follows because for fixed $b \in B$ the map
\[
 \bigvee_{t'(d)=b} \Big( \bigwedge_{n'(e)=d} E \Big) \xToo{\mu_\otherE^D} \bigvee_{t'(d)=b} E
\]
is given on a wedge summand $d = (b,s)$ by $\mu_A^B$, by stability under pullback.


\end{proof}


\end{document}